\documentclass[reqno,twoside]{amsart}

\usepackage{amsmath,amssymb}
\usepackage[english]{babel}
\usepackage{paralist}
\usepackage[colorlinks=true]{hyperref}
\usepackage{array}
\usepackage[latin1]{inputenc}
\usepackage{xcolor}
\usepackage{esint}
\usepackage{latexsym,amsfonts}
\usepackage{amsthm}
\usepackage{array}
\usepackage{cleveref}
\usepackage{enumitem}
\usepackage{mathtools}
\usepackage{verbatim}
\usepackage{graphics}

\usepackage{tikz}

\numberwithin{equation}{section}

\newtheorem{theorem}{Theorem}[section]
\newtheorem{lemma}[theorem]{Lemma}
\newtheorem{prop}[theorem]{Proposition}
\newtheorem*{Thm}{Theorem}

\theoremstyle{definition}
\newtheorem{remark}[theorem]{Remark}
\newtheorem{example}[theorem]{Example}
\newtheorem{question}[theorem]{Question}

\newcommand{\abs}[1]{\left\vert#1\right\vert}

\newcommand{\R}{\mathbb R}
\newcommand{\N}{\mathbb N}
\newcommand{\Z}{\mathbb Z}

\begin{document}
\title[Singular measures with strong annular decay]{On a class of singular measures satisfying a strong annular decay condition}

\author[Arroyo]{\'Angel Arroyo}
\address{Department of Mathematics and Statistics, University of Jyv\"askyl\"a, PO~Box~35, FI-40014 Jyv\"askyl\"a, Finland}
\email{angel.a.arroyo@jyu.fi}

\author[Llorente]{Jos\'e G. Llorente}
\address{Departament de Matem\`atiques, Universitat Aut\`onoma de Barcelona, 08193 Bellaterra, Barcelona, Spain}
\email{jgllorente@mat.uab.cat}

\maketitle

\let\thefootnote\relax\footnotetext{\hspace{-7pt}\emph{Keywords:} annular decay condition,doubling measure,  Bernoulli product, metric measure space. 


Partially supported by grants MTM2017-85666-P, 2017 SGR 395.}

\begin{abstract}
A metric measure space $(X, d, \mu )$ is said to satisfy the \textit{strong annular decay condition} if there is a constant $C>0$ such that 
$$
\mu \big( B(x, R) \setminus B(x,r) \big ) \leq C\, \frac{R-r}{R}\, \mu (B(x,R))
$$ 
for each $x\in X$ and all $0<r \leq R$. If $d_{\infty}$ is the distance induced by the $\infty$-norm in $\R^N$, we construct examples of singular measures $\mu$ on $\R^N$ such that $(\R^N, d_{\infty}, \mu )$ satisfies the strong annular decay condition.           
\end{abstract}

\section{Introduction}\label{sec:1}

A metric measure space $(X, d, \mu)$ is a metric space endowed with a positive Borel measure. Hereafter we will assume that $0 < \mu (B) < \infty$ for every ball $B\subset X$. Let $(X,d, \mu)$ be a metric measure space and $\delta \in (0,1]$. We say that  $(X, d, \mu)$ satisfies a $\delta$-\textit{annular decay condition} ($\delta$- ADC) if there exists a constant $C>0$ such that 
\begin{equation} \label{delta-ADC}
\mu \big( B(x,R)\setminus B(x,r) \big) \leq C \Big ( \frac{R-r}{R} \Big)^{\delta} \, \mu ( B(x,R)) 
\end{equation}
for each $x\in X$ and all $0<r \leq R$. Whenever the ambient metric space is fixed we will often say that the measure $\mu$ itself satisfies a $\delta$-ADC. The case $\delta = 1$ is special in some senses. The $1$-ADC is often known as the \textit{strong annular decay condition} in the literature. Observe that the $1$-ADC implies the $\delta$-ADC for any $\delta \in (0,1]$. 

\

The  $\delta$-ADC is closely connected to the doubling property. We say that a positive Borel measure $\mu$  on a metric space $(X,d)$ is \textit{doubling} if there is a constant $D>0$ such that 
\begin{equation*}
\mu (B(x,2R)) \leq D \, \mu (B(x,R))
\end{equation*}
for each $x\in X$ and any $R>0$. 

\

It is easy to see that a measure satisfying a $\delta$-ADC for some $\delta \in (0,1]$ is doubling (with a doubling constant depending on $\delta$ and the constant in the $\delta$-ADC). Conversely, if $(X,d)$ is geodesic then any doubling measure on $X$ satisfies a $\delta$-ADC for some $\delta \in (0,1)$ only depending on the doubling constant (\cite{B}, see also \cite{AL1} for an elementary proof when $X= \R^N$). Thus in a geodesic metric space, the doubling condition is equivalent to the $\delta$-ADC for some $\delta \in (0,1)$.      

\

The $\delta$-ADC was apparently introduced by Colding-Minicozzi in manifolds (\cite{CM}) and, independently, by Buckley (\cite{B}) in metric spaces. In the last years the $\delta$-ADC has been successfully used in several problems of Harmonic Analysis and Geometric Function Theory: when studying  reverse H\"{o}lder inequalities and characterizations of $A_{\infty}$ in metric spaces (see \cite{KiS1,KiS2,KiS3}), Hardy inequalities  and $T(b)$ theorems (\cite{AR}), capacity estimates (\cite{BB}) and also in connection to the regularity  of functions satisfying certain mean value properties (see \cite{AGG, AL2}).

\begin{remark} Observe that if $(X, d, \mu)$ satisfies a  $\delta$-ADC then  $\mu \{ y: d(x,y) = r \} = 0$ for each $x\in X$ and any $r>0$, that is, $\mu$ does not charge mass over spheres. This is the reason why in \eqref{delta-ADC} it is irrelevant whether the balls are open or closed.   
\end{remark}

In what follows we will restrict to the case $X= \R^N$ and $d=d_p$, the induced  distance by the $p$-norm in $\R^N$, for $1 \leq p \leq \infty$.  Occasionally, $B_p(a,r)$ will denote the ball centered at $a\in \R^N$ and radius $r>0$ with respect to $d_p$. Our main concern in this paper is the regularity of measures satisfying a $1$-ADC with respect to the distance $d_{\infty}$. 

\

Regarding the case $0 < \delta < 1$, the existence of singular doubling measures on $\R^N$ is a classical fact (\cite{K,Z}). See \cite{BA,CFK,L} for connections of singular doubling measures to problems in Geometric Function Theory and PDE's. We will provide explicit examples of singular doubling measures in Subsection \ref{sec:2.2}. As for the case $\delta = 1$, the situation is more delicate. It is straightforward to check that if $N=1$ then any positive measure on $\R$ (endowed with any norm-induced distance) satisfies a $1$-ADC if and only if $\mu$ is absolutely continuous with bounded density. However, this is not so clear if $N\geq 2$. In the higher dimensional case, the problem seems to be very sensitive with respect to the geometry of the distance. Our main result is the following. 

\begin{Thm}\label{theorem}
Let $N\geq 2$ and let $d_{\infty}$ be the distance induced by the $\infty$-norm in $\R^N$. Then there exists a singular measure $\mu$ on $\R^N$ such that $(\R^N , d_{\infty}, \mu )$ satisfies  the $1$-annular decay condition. 
\end{Thm}

The organization of the paper is as follows. \Cref{sec:2} describes an special class of probability measures on the cube $Q_0 = [-\frac{1}{2},\frac{1}{2})^N$, the so called Bernoulli products. Subsection \ref{sec:2.1} provides the fundamental properties of Bernoulli products that will be needed in the paper. The material in this subsection is probably known but we have included it for completeness. Subsection \ref{sec:2.2} introduces a particular class of doubling Bernoulli products. The proof of the Theorem is contained in \Cref{sec:3}. Finally, \Cref{sec:4} is devoted to further results, questions  and comments.    

\

Through the paper, $m_N$ will denote Lebesgue $N$-dimensional measure. If $x\in \R^N$ and $r>0$ we write $Q(x,r) = B_{\infty} (x,r) $ to denote the cube centered at $x$ of side $2r$. All the cubes considered  in the paper are assumed to have  faces parallel to the coordinate hyperplanes.

\section{Bernoulli products}\label{sec:2}

\subsection{Basic definitions and properties}\label{sec:2.1}

We introduce a class of Borel probability measures on $Q_0 = [-\frac{1}{2},\frac{1}{2})^N \subset \R^N$ called \textit{Bernoulli products} which we further extend to the whole $\R^N$ by periodization. Apart from Lebesgue measure (which is a "trivial" particular case), we will see that any other Bernoulli product results in a singular measure. 

\

We will need some preliminary notation. Fix an odd integer  $p = 2q+1$  and consider the family $\mathcal{F}_n$ of $p$-adic subcubes of $Q_0$ of the generation $n$, that is
\begin{equation*}
\mathcal{F}_n = \Bigg\{ \prod_{j=1}^N \bigg [-\frac{1}{2} + \frac{k_j -1}{p^n}, \,  -\frac{1}{2} +\frac{k_j}{p^n} \bigg) \ : \  k_j = 1,2,\ldots,p^n \Bigg\}.
\end{equation*}
For each $x\in Q_0$ and any $n = 0,1,2,\ldots $ denote by $Q_n (x) $ the unique cube belonging to $\mathcal{F}_n$ so that $x\in Q_n (x)$. The fact that the division number is odd is not important but it will be technically convenient in the next subsection. 

\

Let $\mathcal{I} = \{ -q,\ldots, -1,0, 1,\ldots, q \}^N $. Note that, whenever $Q_{n-1} \in  \mathcal{F}_{n-1}$, 
the $p^N$ descendants of $Q_{n-1}$ in $\mathcal{F}_n$ can be labeled using $\mathcal{I}$ in a natural way. Indeed, if  
$$
Q_{n-1} = \prod_{j=1}^N \bigg[-\frac{1}{2} + \frac{k_j -1}{p^{n-1}}, \, -\frac{1}{2} + \frac{k_j}{p^{n-1}} \bigg ),
$$
where $1 \leq k_j \leq p^{n-1} $, then we define, for each  $\nu = (\nu_1, \nu_2, \ldots , \nu_N ) \in \mathcal{I}$,
\begin{equation} \label{Q nu}
Q_{n}^{\nu} = \prod_{j=1}^N \bigg[-\frac{1}{2} + \frac{k_j-1}{p^{n-1}} + \frac{q + \nu_j }{p^n}, \, -\frac{1}{2} + \frac{k_j -1}{p^{n-1}} + \frac{q + \nu_j +1}{p^n}\bigg )
\end{equation}
so that $\{Q_n^{\nu} \}_{\nu \in \mathcal{I}}$ is the partition of $Q_{n-1}$ in cubes of $\mathcal{F}_n$. Also, we can uniquely identify each $Q_n \in \mathcal{F}_n$ with a finite sequence $\{ \nu^1 ,\ldots, \nu^n  \} \subset \mathcal{I}$  and with the corresponding sequence of ancestors $Q_n \subset Q_{n-1} \subset \cdots\subset Q_1 \subset Q_0$ such that $Q_k \in \mathcal{F}_k$ and, for any given  $Q_{k-1}$,
then $Q_k = Q_k^{\nu^k}$ in the sense of \eqref{Q nu}.      

\

Now suppose that $\mathcal{P} = \{p_{\nu} \}_{\nu \in \mathcal{I}}$ is a probability distribution on $\mathcal{I}$ or, equivalently, a probability distribution on the $p^N$ subcubes of $Q_0$ in $\mathcal{F}_1$. We claim that $\mathcal{P}$ induces a Borel probability measure $\mu_{\mathcal{P}}$ on $Q_0$. We define first $\mu_{\mathcal{P}}$ on the family $\displaystyle \bigcup_n \mathcal{F}_n$ of all $p$-adic subcubes of $Q_0$. Assume that $\mu_{\mathcal{P}}(Q_{n-1})$ is already known, where $Q_{n-1} \in \mathcal{F}_{n-1}$. Then for each $\nu \in \mathcal{I}$ and each descendant $Q_{n}^{\nu} \subset Q_{n-1}$ with $Q_{n}^{\nu} \in \mathcal{F}_n$, set 
\begin{equation}\label{def mu-P}
\mu_{\mathcal{P}} (Q_n ^{\nu}) = p_{\nu} \, \mu_{\mathcal{P}} (Q_{n-1}).
\end{equation}
By Kolmogorov's extension theorem (\cite[Ch. II, \S 3, Thm. 4]{S}) the assignment \eqref{def mu-P} extends to a Borel probability measure on $Q_0$ that we still denote by $\mu_{\mathcal{P}}$. Hereafter we will refer to $\mu_{\mathcal{P}}$ as the \textit{Bernoulli product} associated to the (finite) probability distribution $\mathcal{P}$ on $\mathcal{I}$. Once $\mu_{\mathcal{P}}$ has been defined on $Q_0$ we extend it periodically to $\R^N$ by setting 
$\mu_\mathcal{P}(E+z)=\mu_\mathcal{P}(E)$ for $E\subset Q_0$ and $z\in\Z^N$. 
See \cite{H2} for a general discussion of Bernoulli products and \cite{H1,LN} for applications to Geometric Function Theory.   

\

Define the random variable $X_1$ in $Q_0$ by setting 
$$
X_1 (x) =  \sum_{\nu \in \mathcal{I}} \log p_{\nu} \, \mathbf{1}_{Q^{\nu}_1}(x).
$$
We introduce some more notation now. Let $\hat{T}: [-\frac{1}{2},\frac{1}{2}) \to [-\frac{1}{2},\frac{1}{2})$ be the shift transformation given by $\hat{T}t = pt+q -[p(t +\frac{1}{2})]$, where $[.]$ denotes integer part. Now, for $x= (x_1,\ldots,x_N) \in Q_0$ define $T: Q_0 \to Q_0$ by 
$$
Tx = (\hat{T}x_1,\ldots, \hat{T}x_N).
$$
Finally, we define the random variables $X_n (x) = X_1 ( T^{n-1}x)$ for all $n\in \N$. Observe that $\{ X_n \}$ is a sequence of independent (respect to $m_N$), identically distributed random variables with a common distribution given by 
$$
\mathbb{P} (X = s) = \sum_{\log p_{\nu} = s} p_{\nu}.
$$

It turns out that the expression of $\mu_{\mathcal{P}}$ takes a very simple form on $p$-adic subcubes of $Q_0$ when using the random variables $\{ X_n \}$. Indeed, by construction, 
\begin{equation}\label{Xn}
X_n (x) = \log \frac{\mu_{\mathcal{P}}(Q_n (x))}{\mu_{\mathcal{P}}(Q_{n-1} (x))},
\end{equation}
so that
\begin{equation} \label{exp}
\mu_{\mathcal{P}}(Q_n (x)) = e^{S_n (x)},
\end{equation}
where 
\begin{equation} \label{Sn}
S_n (x) = X_1 (x) + \ldots +X_n (x)
\end{equation}
for all $n\in \N$ and each $x\in Q_0$. 

\begin{example}\label{basicex}
To fix ideas, take $N=2$ and $p=3$ (this corresponds to the case when $Q_0 = [-\frac{1}{2},\frac{1}{2})^2$ is divided in $9$ triadic subsquares). The index set $\mathcal{I}$ consists of all the pairs $(i,j)$, with $i,j \in \{-1,0,1\}$,  labeling the $9$ triadic subsquares of $Q_0$ as in \eqref{Q nu}. Choose positive numbers $a_0$, $a_1$ and $a_2$ such that $4a_2 +4a_1 + a_0 =1 $. Let us call group $\bf{2}$ to the four subsquares on the corners (labeled by $(\pm 1, \pm1 ) $), group $\bf 1$ to the four subsquares labeled by $(0, \pm 1 )$, $(\pm 1 , 0)$ and, finally, group $\bf 0$ to the remaining central subsquare, labeled by $(0,0)$.  We now assign measure $a_2$ to each subsquare of group $ \bf 2$, measure $a_1$ to each subsquare of group $ \bf 1$ and measure $a_0$ to the subsquare in group $ \bf 0$ (see \Cref{fig}).
\begin{figure}[h]
\begin{tikzpicture}[scale=0.5]
 
\draw (0,0) rectangle (6,6);
\draw (0,2) -- (6,2);
\draw (0,4) -- (6,4);
\draw (2,0) -- (2,6);
\draw (4,0) -- (4,6);
\draw [scale=1](1,1) node [scale=1]{$a_2$};
\draw [scale=1](1,3) node [scale=1]{$a_1$};
\draw [scale=1](1,5) node [scale=1]{$a_2$};
\draw [scale=1](3,1) node [scale=1]{$a_1$};
\draw [scale=1](3,3) node [scale=1]{$a_0$};
\draw [scale=1](3,5) node [scale=1]{$a_1$};
\draw [scale=1](5,1) node [scale=1]{$a_2$};
\draw [scale=1](5,3) node [scale=1]{$a_1$};
\draw [scale=1](5,5) node [scale=1]{$a_2$};
 
\end{tikzpicture}
\caption{\label{fig}}
\end{figure}
This defines a particular probability distribution $\mathcal{P}$ on $\mathcal{I}$. Then keep this pattern and repeat it independently in each generation. The random variables $\{X_n \}$ given by \eqref{Xn} are independent and identically distributed (with respect to $m_2$)
and the measure $\mu_{\mathcal{P}}$ acting on triadic sub-squares can be written in the following way
$$
\mu_{\mathcal{P}} (Q_n (x)) = a_2 ^{ A_n^2 (x)} a_1 ^{A_n^1 (x)} a_0 ^{A_n^0 (x)},
$$
where $A_n^{j} (x)$  is the number of $k$'s ($k=1,\ldots,n$) such that $Q_k (x)$ belongs to group $\bf j$ ($\bf{j} = \bf{0},\bf{1},\bf{2}$). Observe that $S_n = X_1 +\cdots+X_n = (\log a_2 ) \, A_n^2  + (\log a_1 ) \, A_n^1 + (\log a_0 ) \, A_n^0 $ in this case.  
\end{example}

Coming back to the general case, we claim now that $\mu_{\mathcal{P}}$ is singular unless $\mathcal{P}$ is the uniform probability distribution on $\mathcal{I}$, in which case  $ \mu_{\mathcal{P}} = m_N$. To see this, remind that, by \eqref{exp} and \eqref{Sn} we have, for each $Q_n \in \mathcal{F}_n$, 
\begin{equation}\label{density}
\log \bigg ( \frac{\mu_{\mathcal{P}} (Q_n )}{m_N (Q_n )} \bigg ) = \log ( p^{nN} e^{S_n} ) = n\bigg ( N \log p  + \frac{S_n}{n} \bigg ).
\end{equation}
We invoke now the Law of Large Numbers for the random variables $\{ X_n \}$ and deduce that 
\begin{equation*}
\frac{S_n }{n} \to \mathbb{E}[X_1] = p^{-N}\sum_{\nu \in \mathcal{I}} \log p_{\nu} \qquad (m_N\mbox{-a.e.}),
\end{equation*}
so the expression inside brackets  at the right-hand side of \eqref{density} converges to 
\begin{equation}\label{concav}
N \log p + \sum_{\nu \in \mathcal{I}} \log p_{\nu} ,
\end{equation}
(remind that $\mathrm{card} (\mathcal{I}) = p^N$). Now we claim that \eqref{concav} is strictly negative unless all $p_{\nu}$'s are equal. This is a consequence of the fact that if $k\in \N$ and $x_1,\ldots,x_k$ are positive numbers such that $ \displaystyle \sum_{i=1}^k x_i = 1$, then $\displaystyle \sum_{i=1}^k \log x_i  \leq - k\log k
$, with equality if and only if all the $x_i$'s are equal. Therefore, unless all the $p_{\nu}$'s are equal,  
$$
\log \bigg ( \frac{\mu_{\mathcal{P}} (Q_n (x) )}{m_N (Q_n (x) )} \bigg ) \to -\infty
$$
for $m_N$-a.e. $x\in Q_0$. This shows that $\mu_{\mathcal{P}}$ is singular, with the only exception in which $\mathcal{P}$ is the uniform distribution and therefore $\mu_{\mathcal{P}} = m_N$.

\begin{remark}
An alternative to the argument above is based in the fact that the random variables $\{ X_n \}$ are also independent and identically distributed with respect to the probability $\mu_{\mathcal{P}}$ itself. In this case 
$$
\mathbb{E}_{\mu_{\mathcal{P}}}[X_1] = \sum_{\nu \in \mathcal{I}}p_{\nu} \, \log p_{\nu} = -h(\mathcal{P}),
$$
where $h(\mathcal{P})$ denotes the \textit{entropy} of the probability distribution $\mathcal{P}$. Then \eqref{density} holds in the same way and the Law of the Large Numbers gives now that 
$$
\frac{S_n}{n} \to -h(\mathcal{P})  \quad (\mu_{\mathcal{P}}\mbox{-a.e.})
$$
as $n\to \infty$, so
$$
N\log p + \frac{S_n}{n} \to N\log p - h(\mathcal{P}) \quad (\mu_{\mathcal{P}}\mbox{-a.e.}).
$$
Now it follows from the basic estimate for  the entropy that  $N\log p - h(\mathcal{P}) \geq 0 $ with equality if and only if all the $p_{\nu}$'s are equal. Therefore, unless $\mu_{\mathcal{P}} = m_N$, we have 
$$
\log \bigg ( \frac{\mu_{\mathcal{P}} (Q_n (x) )}{m_N (Q_n (x) )} \bigg ) \to +\infty
$$
for $\mu_{\mathcal{P}}$-a.e. $x\in Q_0$, hence $\mu_{\mathcal{P}}$ is singular. The measure $\mu_{\mathcal{P}}$ is an example of a \textit{unidimensional} measure, with dimension $ \alpha(\mathcal{P}) = (\log p)^{-1} h(\mathcal{P}) $. Observe that $0 < \alpha(\mathcal{P}) \leq N $ and $\alpha(\mathcal{P}) = N$ if and only if $\mu_{\mathcal{P}}$ coincides with $m_N$. (See \cite{H1,L,LN} for further details). 
\end{remark}

\begin{remark} It can be shown that he probabilities $\mu_{\mathcal{P}}$ are mutually singular.
\end{remark}

\begin{remark} We could have also invoked to the Ergodic Theorem applied to the shift transformation $T$  instead of the Law of the Large Numbers. 
\end{remark}

\subsection{A particular class of doubling Bernoulli products}\label{sec:2.2}

Let $Q_0=[-\frac{1}{2},\frac{1}{2})^N$, $p=3$, $\mathcal{I}=\{-1,0,1\}^N$ and $\mathcal{F}_n$ be the family of triadic subcubes of $Q_0$ of the $n$-th generation. For any $\nu=(\nu_1,\nu_2,\ldots,\nu_N)\in\mathcal{I}$, we define the \textit{length} of $\nu$ as the natural number
$$
		\abs{\nu}:\,=\sum_{j=1}^N\abs{\nu_j}\in\{0,1,2,\ldots,N\}.
$$
We will construct a probability distribution $\mathcal{P}= \{p_{\nu} \}_{\nu \in \mathcal{I}}$ on $\mathcal{I}$ assigning the same probability to vectors in $\mathcal{I}$ of the same length. That is, for each $\nu\in\mathcal{I}$,
\begin{equation}\label{p-nu=a-nu}
		p_\nu=a_{\abs{\nu}},
\end{equation}
where $a_0,a_1,\ldots,a_N\in(0,1)$ are fixed probabilities satisfying
$$
		\sum_{k=0}^N \mathrm{card}\{\nu\in\mathcal{I}\ : \ \abs{\nu}=k\}\cdot a_k=1.
$$
Each of the firsts factors in the sum above can be computed following a simple combinatorial argument, and thus the equation becomes
\begin{equation*}
		\sum_{k=0}^N 2^k\binom{N}{k}\cdot a_k=1.
\end{equation*}

Constructing the corresponding probability measure on $Q_0$ as in the previous subsection, we obtain a probability measure $\mu=\mu_\mathcal{P}$ on $Q_0$, which is singular unless $a_0=a_1=\ldots=a_N=3^{-N}$, in which case $\mu = m_N$. Observe in particular that $\mu$ is symmetric with respect to each coordinate hyperplane. This means that for every set $A\subset Q_0$ it holds $\mu(A)=\mu(\widetilde{A})$, where $\widetilde{A}$ is any reflection of $A$ with respect to any coordinate hyperplane.
\\

In what follows, we show that the measure $\mu$ defined in this subsection satisfies the doubling condition. This property will be one of the tools in the proof of the main result of this article. The key is the following quasi-symmetry property of $\mu$, which controls the ratio of  the measures of two contiguous cubes of the same generation.

\begin{lemma}\label{lemma}
Let $Q_n,\widetilde Q_n\in\mathcal{F}_n$ be any two cubes of the $n$-th generation sharing one of their $(N-1)$-dimensional faces. Then
\begin{equation}\label{Q1-Q2}
		\frac{\mu(Q_n)}{\mu(\widetilde Q_n)}\geq a_\mathrm{min},
\end{equation}
where $a_\mathrm{min}:\,=\mathrm{min}\{a_0,a_1,\ldots,a_N\}\in(0,1)$.
\end{lemma}

\begin{proof}
Without loss of generality, we can assume that both $Q_n$ and $\widetilde Q_n$ are contained in $Q_0$. Otherwise, if $Q_n\subset Q_0$ and $\widetilde Q_n\not\subset Q_0$, since $\mu_\mathcal{P}$ is extended from $Q_0$ to $\R^n$ periodically, and due to the symmetry of the measure with respect to the coordinate axis, both $Q_n$ and $\widetilde Q_n$ have the same measure and \eqref{Q1-Q2} follows immediately.

Since $Q_n,\widetilde Q_n\in\mathcal{F}_n$, recalling notation \eqref{Q nu}, there exist $Q_{n-1},\widetilde Q_{n-1}\in\mathcal{F}_{n-1}$ and $\nu,\widetilde\nu\in\mathcal{I}$ such that
$$
		Q_n=Q_n^\nu\subset Q_{n-1} \quad \quad \mbox{ and } \quad \quad \widetilde Q_n=\widetilde Q_n^{\widetilde\nu}\subset \widetilde Q_{n-1}.
$$
Then, by \eqref{def mu-P}, we have
\begin{equation}\label{measure Q1/Q2}
		\frac{\mu(Q_n)}{\mu(\widetilde Q_n)}=\frac{a_{\abs{\nu}}}{a_{\abs{\widetilde\nu}}}\cdot\frac{\mu(Q_{n-1})}{\mu(\widetilde Q_{n-1})}.
\end{equation}
Then, either $Q_{n-1}=\widetilde Q_{n-1}$ or $Q_{n-1}\neq \widetilde Q_{n-1}$. In the first case, this means that $\abs{\nu}\neq\abs{\widetilde\nu}$, and \eqref{measure Q1/Q2} yields
\begin{equation}\label{case 1}
		\frac{\mu(Q_n)}{\mu(\widetilde Q_n)}=\frac{a_{\abs{\nu}}}{a_{\abs{\widetilde\nu}}}\geq a_\mathrm{min},
\end{equation}
and thus the proof is finished in this case. On the other hand, if $Q_{n-1}\neq \widetilde Q_{n-1}$, since $Q_n$ and $\widetilde Q_n$ are contiguous, then $Q_{n-1}$ and $\widetilde Q_{n-1}$ are also contiguous and there exists $j_0\in\{1,2,\ldots,N\}$ such that $\nu_j=\widetilde\nu_j$ for all $j\neq j_0$ and $\nu_{j_0}=-\widetilde\nu_{j_0}$. In particular, $\abs{\nu}=\abs{\widetilde\nu}$ and \eqref{measure Q1/Q2} becomes
\begin{equation}\label{case 2}
		\frac{\mu(Q_n)}{\mu(\widetilde Q_n)}=\frac{\mu(Q_{n-1})}{\mu(\widetilde Q_{n-1})}.
\end{equation}
Now, let $m\leq n-1$ be the generation of the nearest common ancestor of $Q_n$ and $\widetilde Q_n$. That is, the largest $m\in\N$ for which there exists $Q_m\in\mathcal{F}_m$ such that $Q_n\cup \widetilde Q_n\subset Q_m$. Then, we can construct the following sequences of nested contiguous cubes,
\begin{equation*}
\begin{split}
		& Q_m\supset Q_{m+1}\supset\cdots\supset Q_{n-1}\supset Q_n, \\
        & Q_m\supset \widetilde Q_{m+1}\supset\cdots\supset \widetilde Q_{n-1}\supset \widetilde Q_n,
\end{split}
\end{equation*}
and repeat \eqref{case 2} $n-m-1$ times to get
$$
		\frac{\mu(Q_n)}{\mu(\widetilde Q_n)}=\frac{\mu(Q_{m+1})}{\mu(\widetilde Q_{m+1})}.
$$
Since $Q_{m+1}\cup \widetilde Q_{m+1}\subset Q_m$, we apply \eqref{case 1} and we get \eqref{Q1-Q2}.
\end{proof}

\begin{prop}\label{result-doubling}
The measure $\mu$ is doubling. That is, there exists a constant $D=D(N,a_\mathrm{min})>0$ such that
\begin{equation}\label{doubling}
		\mu(Q(x,2r))\leq D\, \mu(Q(x,r))
\end{equation}
for each $x\in\R^n$ and $r>0$.
\end{prop}

\begin{proof}
If $r\gtrsim 1$, then it follows from the periodicity of the construction of $\mu$ that $\mu(Q(x,r))\simeq r^N$. This implies the doubling condition for large $r$. For small $r$, we can assume also from periodicity and symmetry that $Q(x,2r)\subset Q_0$. Therefore, we can start by assuming that $0<r\leq\frac{1}{12}$ and fix $n=\max\{m\in\N \, : \, 4r\leq 3^{-m}\}$. Then
$$
		2r\leq\frac{1}{2\cdot 3^n}
$$
and $Q(x,2r)$ is contained in the union of (at most) $2^N$ cubes of generation $n$ sharing at least one vertex with $Q_n(x)$. Let $Q_n$ be one of such cubes. Then, from \Cref{lemma} and an elementary chain argument it follows that $\mu(Q_n)\leq a_\mathrm{min}^{-N}\,\mu(Q_n(x))$. Therefore, we obtain
$$
		\mu(Q(x,2r))\leq \frac{2^N}{a_\mathrm{min}^N}\,\mu(Q_n(x)).
$$
On the other hand, from the choice of $n\in\N$, it follows also that,
$$
		\frac{1}{3^{n+3}}<r,
$$
and thus $Q_{n+3}(x)\subset Q(x,r)$. Then, since $Q_{n+3}(x)\subset Q_n(x)$ and from the self-similarity of the construction of $\mu$, we get
$$
		\mu(Q(x,r))\geq\mu(Q_{n+3}(x))\geq a_\mathrm{min}^3\,\mu(Q_n(x)).
$$
Hence, \eqref{doubling} follows with $D=2^N a_\mathrm{min}^{-N-3}$.
\end{proof}

In order to ensure that the corresponding measure $\mu$ satisfies the $1$-ADC, we need to impose first some conditions on the coefficients $a_k$. For that reason, let us split $Q_0$ into three sets $\Gamma^{-1}$, $\Gamma^0$ and $\Gamma^1$ as follows:
$$
		\Gamma^i=\bigcup\left\{Q_1^\nu \ : \ \nu_1=i\right\},
$$
for $i\in\{-1,0,1\}$. Note that, by the symmetry of $\mu$, $\Gamma^{-1}$ and $\Gamma^1$ have the same measure, then $2\mu(\Gamma^1)+\mu(\Gamma^0)=1$. In addition, we impose that the measure of each $\Gamma^i$ is exactly equal to $1/3$, that is,
\begin{equation}\label{condition-mu-P}
		\mu(\Gamma^1)=\frac{1}{3} \quad \quad \mbox{ and } \quad \quad \mu(\Gamma^0)=\frac{1}{3}.
\end{equation}
Since the probability distribution of the cubes in $\mathcal{F}_1$ coincides with the probability distribution $\mathcal{P}=\{p_\nu=a_{\abs{\nu}}\}_{\nu\in\mathcal{I}}$, the above conditions can be expressed in terms of the coefficients $a_k$:
\begin{equation*}
\begin{split}
		\mu(\Gamma^i)
        = ~& \sum_{\nu\in\mathcal{I} \, : \, \nu_1=i}\mu(Q_1^\nu) 
       = \mu(Q_0)\sum_{\nu\in\mathcal{I} \, : \, \nu_1=i} p_{\nu} \\
        = ~& \sum_{k=0}^N \mathrm{card}\{\nu\in\mathcal{I}\ : \ \nu_1=i \ \mbox{ and }\ \abs{\nu}=k\}\cdot a_k,
\end{split}
\end{equation*}
where \eqref{def mu-P} and \eqref{p-nu=a-nu} have been used in the second and the third equalities, respectively. Then, by some elementary combinatorics,  \eqref{condition-mu-P} is equivalent to 
\begin{equation}\label{eqs}
		\begin{cases}
		\displaystyle\ \sum_{k=0}^{N-1} 2^k\binom{N-1}{k}\cdot a_k=\frac{1}{3} & \\ & \\
        \displaystyle\ \sum_{k=0}^{N-1} 2^k\binom{N-1}{k}\cdot a_{k+1}=\frac{1}{3} &
        \end{cases}
\end{equation}

\ 

Let $S\subset\R^{N+1}$ be the set of all possible choices of the vector  $(a_0, a_1 ,\ldots,a_N ) \in (0,1)^{N+1}$ satisfying \eqref{eqs}, that is
$$
		S=\{(a_0,a_1,\ldots,a_N) : \mbox{\eqref{eqs} holds and } 0<a_k<1 \mbox{ for } k=0,1,2,\ldots,N\}.
$$
This set is nonempty since $(3^{-N},\ldots,3^{-N})\in S$. Note also  that the equations in \eqref{eqs} define two $N$-dimensional hyperplanes in $\R^{N+1}$ and they both contain $S\neq\emptyset$, then its intersection is an $(N-1)$-dimensional affine subspace in $\R^{N+1}$. Let us call it $\pi$. Moreover, since $S$ is the intersection in $\pi$ of the open sets $\{0<a_k<1\}$ for $k=0,1,2,\ldots,N$, then $S$ is an open set with respect to the topology induced in $\pi$. Therefore, there exists different choices of the coefficients in $S$ apart from $(3^{-N},\ldots,3^{-N})$. Thus, the construction provides examples of doubling, singular Bernoulli products satisfying \eqref{eqs}.

\begin{example}
For $N=2$, the equations in \eqref{eqs} define the line
$$
		\left\{\left(\frac{1}{9}-4t,\frac{1}{9}+2t,\frac{1}{9}-t\right)\ : \ t\in\R\right\},
$$
and since $0<a_k<1$ for $k=0,1,2$, then
$$
		\left.\begin{array}{l}
        a_0=\dfrac{1}{9}-4\varepsilon \\ \\
        a_1=\dfrac{1}{9}+2\varepsilon \\ \\
        a_2=\dfrac{1}{9}-\varepsilon
        \end{array}\right\}\quad \mbox{ for } -\frac{1}{18}<\varepsilon<\frac{1}{36}.
$$
\end{example}

\section{Proof of the Theorem}\label{sec:3}

In this section, we show that, given $a_0,a_1,\ldots,a_N\in(0,1)$ such that \eqref{eqs} holds, the corresponding measure $\mu=\mu_\mathcal{P}$ satisfies the $1$-ADC with respect to $d_\infty$, i.e.,
$$
		\frac{\mu(Q(x,R)\setminus Q(x,r))}{\mu(Q(x,R))} \lesssim \frac{R-r}{R},
$$
for $0<r<R$. Here, the notation $a\lesssim b$ means that there exists a constant $C>0$ such that $a\leq Cb$. In addition, we write $a\simeq b$ if $a\lesssim b$ and $b\lesssim a$.
\\

We begin introducing some notation. For $x=(x_1,x_2,\ldots,x_N)\in\R^N$ and  $r>0$ we write 
$$
		Q(x,r)= B_{\infty} (x,r) = \prod_{j=1}^N[x_j-r,x_j+r).
$$
In what follows, we say that a set $\Gamma\subset Q(x,r)$ is a \textit{coordinate strip of parameters $2r$ and $t-s>0$} if there exists $m\in\{1,2,\ldots,N\}$ such that
\begin{equation*}
		\Gamma=\prod_{j=1}^{m-1}[x_j-r,x_j+r)\times[s,t)\times\prod_{j=m+1}^N[x_j-r,x_j+r)
\end{equation*}
for some $x_m-r\leq s<t \leq x_m+r$. Moreover, we say that $\Gamma \subset Q_0$ is a \textit{triadic strip of generation $n\in\N$} if it is a coordinate strip of parameters $1$ and $3^{-n}$  and can be expressed as the union of cubes in $\mathcal{F}_{n}$. That is, if there exists $m\in\{1,2,\ldots,N\}$ such that
\begin{equation*}
		\Gamma=\left[-\frac{1}{2},\frac{1}{2}\right)^{m-1}\times\left[-\frac{1}{2}+\frac{k-1}{3^n},-\frac{1}{2}+\frac{k}{3^n}\right)\times\left[-\frac{1}{2},\frac{1}{2}\right)^{N-m},
\end{equation*}
for some $k\in\{1,2,\ldots,3^n\}$.
\\

The proof relies on estimating for the measure of the annular set $Q(x,R)\setminus Q(x,r)$, which can be seen as the union of $2N$ coordinate strips of parameters $2R$ and $R-r$. For that reason, we focus on giving an estimate for the measure of each of these strips. This is stated in the following lemma, which is the main result of this section.

\begin{lemma}\label{key lemma}
Let $\Gamma \subset Q(x,R)$ be a coordinate strip of parameters $2R$ and $h\in (0,2R)$. Then there is a constant $C=C(N,a_\mathrm{min})$ such that
$$
		\frac{\mu(\Gamma)}{\mu(Q(x,R))}\leq C\,\frac{h}{R}.
$$
\end{lemma}

\

Applying this result to each coordinate strip $\Gamma\subset Q(x,R)\setminus Q(x,r)$ of parameters $2R$ and $R-r$, we get
$$
		\frac{\mu(Q(x,R)\setminus Q(x,r))}{\mu(Q(x,R))}\leq 2^NC\ \frac{R-r}{R},
$$
which proves the Theorem.
\\

In order to prove \Cref{key lemma}, we first need to control the measure of \textit{small} coordinate strips contained in $Q_0$. The following result is a particular version of \Cref{key lemma} that covers this case.

\begin{lemma}\label{small strip}
Let $n\geq 0$ and $Q_n\in\mathcal{F}_n$. Suppose that $\Gamma\subset Q_n$ is a coordinate strip of parameters $3^{-n}$ and $h\in(0,3^{-n})$. Then
$$
		\frac{\mu(\Gamma)}{\mu(Q_n)}=3^nh.
$$
\end{lemma}

\begin{proof}
Let us first assume that $\Gamma=\Gamma_n\subset Q_0$ is a triadic strip of generation $n\in\N$. We show that
\begin{equation}\label{lemma triadic strips}
		\mu(\Gamma_n)=\frac{1}{3^n}.
\end{equation}
For $n=1$, this follows directly from condition \eqref{condition-mu-P}. For $n\geq 2$, let $\Gamma_{n-1}\subset Q_0$ be the unique triadic strip of generation $n-1$ that contains $\Gamma_n$. Then
$$
		\mu(\Gamma_n)=\frac{1}{3}\,\mu(\Gamma_{n-1}).
$$
Indeed, since $\Gamma_{n-1}$ is a triadic strip of generation $n-1$, for every $Q_{n-1}\in\mathcal{F}_{n-1}$ such that $Q_{n-1}\subset\Gamma_{n-1}$, then $\mu(\Gamma_n\cap Q_{n-1})=\frac{1}{3}\,\mu(Q_{n-1})$ by self-similarity and thus
$$
		\mu(\Gamma_n)=\sum_{Q_{n-1}\subset\Gamma_{n-1}}\mu(\Gamma_n\cap Q_{n-1})=\frac{1}{3}\sum_{Q_{n-1}\subset\Gamma_{n-1}}\mu(Q_{n-1})=\frac{1}{3}\,\mu(\Gamma_{n-1}).
$$
Therefore, \eqref{lemma triadic strips} follows by induction on $n$.

Now, suppose that $\Gamma\subset Q_0$ is a coordinate strip of parameters $1$ and $h\in (0,1)$ and write $\Gamma$ as the disjoint union of triadic strips. Then, applying \eqref{lemma triadic strips} to each triadic strip in the decomposition and summing their measures, it follows that
\begin{equation}\label{strip lemma}
		\mu(\Gamma)=h.
\end{equation}

Finally, suppose that $\Gamma \subset Q_n$ is a coordinate strip of parameters $3^{-n}$ and $h\in(0,3^{-n})$. By the self-similarity of the construction of the measure, we have
$$
		\mu(\Gamma)
        =\mu(Q_n)\,\mu(3^n\Gamma).
$$
Note that $3^n\Gamma$ is a coordinate strip of parameters $1$ and $3^nh$ in some cube of length $1$ in $\Z^N+Q_0$. Thus, since $\mu$ is extended periodically outside $Q_0$ to the whole $\R^N$, we get from \eqref{strip lemma} that $\mu(3^n\Gamma)=3^n h$.
\end{proof}

\begin{proof}[Proof of \Cref{key lemma}]
Assume first that $R\gtrsim 1$, say $R\simeq m$, for some $m\in\N$. Assume also that $h\leq 1$ (the case $h\geq 1$ can be treated in a similar way). Then $\Gamma$ can be seen as the union of $\sim\! m^{N-1}$ coordinate strips $\Gamma_j$ of parameters $1$ and $h$. Then, by \Cref{small strip},
$$
		\mu(\Gamma) \simeq \sum_j\mu(\Gamma_j) \lesssim m^{N-1}h \simeq \frac{h}{R}\, R^N \simeq \frac{h}{R}\,\mu(Q(x,R)).
$$
Now, assume that $R\lesssim 1$. By the periodicity of $\mu$, we can also assume that $\Gamma\subset Q_0$. In particular, suppose that
\begin{equation}\label{R approx 3^n}
		\frac{1}{2\cdot 3^{n+1}}< R\leq\frac{1}{2\cdot 3^n}
\end{equation}
holds for some $n\in\N$. Consider $\mathcal{Q}\subset\mathcal{F}_n$ the family of cubes $Q_n$ such that $Q_n\cap Q(x,R)\neq\emptyset$. Then, $\mathcal{Q}$ has at most $2^N$ elements and, in particular, $Q_n(x)\in\mathcal{Q}$. By the doubling property \eqref{doubling}, we have that
\begin{equation}\label{comparison of cubes}
		\mu(Q(x,R))\simeq\mu(Q_n)
\end{equation}
for every $Q_n\in\mathcal{Q}$. Now, let $\Gamma\subset Q(x,R)$ be a coordinate strip of parameters $2R$ and $R-r$. Since $Q(x,R)$ is contained in the union of all cubes $Q_n$ in $\mathcal{Q}$, we have in particular that
$$
		\Gamma\subset\bigcup_{Q_n\in\mathcal{Q}}\Gamma_{Q_n},
$$
where each $\Gamma_{Q_n}\subset Q_n$ is a coordinate strip parameters $3^{-n}$ and $R-r$. Therefore, again by \Cref{small strip}, we obtain
$$
		\mu(\Gamma)\leq\sum_{Q_n\in\mathcal{Q}}3^n(R-r)\mu(Q_n).
$$
Recalling \eqref{R approx 3^n} and \eqref{comparison of cubes} we get
$$
		\frac{\mu(\Gamma)}{\mu(Q(x,R))}\lesssim 2^N\ \frac{R-r}{R},
$$
and the proof is finished.
\end{proof}

\section{Further results, questions  and comments}\label{sec:4}

Take $N=2$ and let $\mu$ be the measure introduced in \Cref{basicex}, corresponding to the choices  $a_2 = \frac{1}{9} - \varepsilon$, $a_1 = \frac{1}{9} + 2\varepsilon$ and $a_0 = \frac{1}{9} - 4\varepsilon$ for some $0 < \varepsilon < \frac{1}{36}$. Then, as has been established in \Cref{sec:3}, $(\R^2 , d_{\infty}, \mu)$ satisfies the $1$-ADC. However, we will see in the following proposition that $(\R^2 , d_1 , \mu)$ does not satisfy the $1$-ADC. Hence, the $1$-ADC can discriminate between different norm-induced distances. 
 
\begin{prop} 
If $\mu $ is as above, then $(\R^2 , d_1, \mu)$ does not satisfy the $1$-ADC.   
\end{prop}

\begin{proof}
We first define the cubes 
$$
Q^n_0 = \bigg[-\frac{1}{2 \cdot 3^n}, \frac{1}{2\cdot 3^n }\bigg) \times \bigg [-\frac{1}{2}, -\frac{1}{2} + \frac{1}{3^n} \bigg )  
$$
and $Q^n_j = Q^n_0 + j\,3^{-n}(1,1)$, where $j = 1,2,\ldots,3^n-1$. 
Now let $\displaystyle \mathcal{C}_n = \bigcup_{j=0}^{3^n -1} Q^n_j$. Observe that $\mathcal{C}_n$ is a chain of $3^n$ cubes in $\mathcal{F}_n$ which is "parallel" to the vector $(1,1)$ in the sense that the union of the diagonals of the cubes $\{ Q^n_j \}$ consists exactly on the segment joining the points $  ( -\frac{1}{2\cdot 3^n}, -\frac{1}{2})$ and $ ( 1- \frac{1}{2\cdot 3^n}, \frac{1}{2})$.

From a recursive argument, whose details we omit here, we claim that  
\begin{equation}\label{chain2}
\mu (\mathcal{C}_n ) = (2b+a)^n = \Big(\frac{1}{3} +3\varepsilon \Big)^n.
\end{equation}
Consider now the trapezoid $T_n$ of vertex $(1, \frac{1}{2})$, $(1-\frac{1}{3^n} , \frac{1}{2})$, $(0, -\frac{1}{2})$, $(0, -\frac{1}{2} + \frac{1}{3^n} )$. From the doubling property, there is a fixed constant $C>0$ such that $ \mu (T_n ) \geq C \mu (\mathcal{C}_n )$. To finish the proof, observe that 
$$
T_n \subset B_1 \bigg(\Big(0,\frac{1}{2}\Big), 1\bigg) \setminus B_1 \bigg(\Big(0,\frac{1}{2}\Big), 1-\frac{1}{3^n} \bigg).
$$
If $R= 1$ and $r_n= 1- \frac{1}{3^n}$ then, from \eqref{chain2},
$$
\mu \bigg( B_1 \bigg(\Big(0,\frac{1}{2}\Big),R\bigg) \setminus B_1 \bigg(\Big(0,\frac{1}{2}\Big),r_n\bigg)\bigg) \geq  C (1 +9\varepsilon)^n \, \frac{R-r_n}{R}  ,  
$$
which shows that $(\R^2 , d_1 , \mu ) $ does not satisfy the $1$-ADC.  
\end{proof}

\begin{question}
Can the construction in the Theorem be modified to  construct examples of singular measures $\mu$ on $\R^N$ such that $(\R^N , d_2 , \mu )$ satisfies the strong annular decay condition? 
\end{question}

\begin{remark}
 It is perhaps illuminating to interpret the $1$-ADC in terms of maximal functions. Let us restrict to the case of the euclidean distance $d_2$. For a given Borel measure $\mu$ on $\R^N$ define the maximal density
$$
\mu^* (x) = \sup_{0< \varepsilon \leq 1 } \frac{\mu (B(x, \varepsilon))}{\varepsilon^n}.
$$
Then an standard covering argument shows that $(\R^N, d_2 ,\mu )$ satisfies the $1$-ADC provided 
$$
\fint_{\partial B(x,R)} \hspace{-0.2cm} \mu^* (y) d\sigma (y) \leq C \frac{\mu (B(x,R))}{R^N},
$$
where $\sigma$ denotes surface measure on balls. 
\end{remark}

\noindent
The strong annular decay condition suggests some questions on weights with a noticeable Harmonic Analysis flavor. We display some of them here. 

\begin{question} Is it possible to characterize the doubling weights  $f>0$ on $\R^N$ satisfying
$$
\fint_{\partial B} \hspace{-0.2cm} f(y) \, d\sigma (y)  \leq C \, \fint_{B} \hspace{-0.1cm} f(x) \, dx 
$$
for any ball $B\subset \R^N$? 
\end{question}

\begin{question}  Denote by $\mathcal{M}$ and $\mathcal{M}_s$ the usual and spherical Hardy-Littlewood maximal operators respectively. How to give alternative characterizations of the (doubling) weights $f>0$ such that 
$$
\mathcal{M}_s f (x) \leq C \mathcal{M}f (x)
$$
uniformly in $x\in \R^N$? 
\end{question}

\end{document}